\documentclass{amsart}
\usepackage{amsmath}
\usepackage{amssymb}
\usepackage{amscd}
\usepackage[all]{xy}
\usepackage{color}
\usepackage{tikz-cd}
\usepackage{ulem}
\usepackage{tikz}
\usepackage{comment}
\mathchardef\mhyphen="2D

\newtheorem{theorem}{Theorem}[section]
\newtheorem{lemma}[theorem]{Lemma}
\newtheorem{corollary}[theorem]{Corollary}
\newtheorem{proposition}[theorem]{Proposition}

\newtheorem{example}[theorem]{Example}
\newtheorem{remark}[theorem]{Remark}
\newtheorem{question}[theorem]{Question}

\begin{document}
\title[On Anderson rings]
{On the transfer of certain ring-theoretic properties in Anderson rings}

\author [H. Baek] {Hyungtae Baek}
\address{(Baek) School of Mathematics,
Kyungpook National University, Daegu 41566,
Republic of Korea}
\email{htbaek5@gmail.com}

\author [J. W. Lim] {Jung Wook Lim}
\address{(Lim) Department of Mathematics,
College of Natural Sciences,
Kyungpook National University, Daegu 41566,
Republic of Korea}
\email{jwlim@knu.ac.kr}

\author [A. Tamoussit] {Ali Tamoussit}
\address{(Tamoussit) Department of Mathematics,
The Regional Center for Education and Training Professions Souss Massa,
Inezgane, Morocco}
\email{a.tamoussit@crmefsm.ac.ma; tamoussit2009@gmail.com}

\thanks{Words and phrases: Anderson ring, Nagata ring, Serre's conjecture ring, Jaffard-like domain, $t$-dimension}

\thanks{$2020$ Mathematics Subject Classification: 13A15, 13B25, 13B30, 13F05.}

\date{\today}

\begin{abstract}
Let $R$ be a commutative ring with unity and let $X$ be an indeterminate over $R$.
The \textit{Anderson ring} of $R$ is defined as the quotient ring of
the polynomial ring $R[X]$ by the set of polynomials that evaluate to $1$ at $0$.
Specifically, the Anderson ring of $R$ is $R[X]_A$, where $A=\{f\in R[X]\mid f(0)=1\}$.
In this paper, we aim to investigate the transfer of various ring-theoretic properties between the ring $R$ and its Anderson ring $R[X]_A$.
Interesting results are established, accompanied by applications and illustrative examples.
\end{abstract}

\maketitle

\section{Introduction}\label{Sec1}

Throughout this paper, $R$ is a commutative ring with identity and
$R[X]$ is the polynomial ring over $R$.
For the sake of clarity, we use $D$ instead of $R$
when $R$ is an integral domain.
Additionally, ${\rm Spec}(R)$ denotes the set of prime ideals of $R$ and
${\rm Max}(R)$ denotes the set of maximal ideals of $R$.

Consider the sets
$A = \{f \in R[X] \,|\, f(0) = 1\}$,
$U = \{f \in R[X] \,|\, f \text{ is monic}\}$ and
$N = \{f \in R[X] \,|\, c(f) = R\}$,
where $c(f)$ is the ideal of $R$ generated by the coefficients of $f$.
It is clear that each of these subsets, $A$, $U$ and $N$,
forms a multiplicative subset of $R[X]$.
This allows us to define the corresponding quotient rings  $R[X]_A$, $R[X]_U$ and $R[X]_N$.

In 1936, the quotient ring $R[X]_N$ was first introduced by Krull  \cite{krull 1936}
and further studied by Nagata in \cite{nagata 1956, nagata 1962}, where it was denoted as $R(X)$.
Subsequently, many ring-theorists have investigated this construction,
and it was later named the {\it Nagata ring} by Querré in \cite{querre 1980}.
In another context, in 1955, Serre \cite{serre 1955} posed the question
for the existence of projective $k[X_1, \dots, X_n]$-modules of finite type that are not free, where $k$ is a field.
In 1976, Quillen introduced the quotient ring $R[X]_U$ and provided a solution to Serre's conjecture
by proving that every finitely generated projective $D[X_1, \dots, X_n]$-module is free for any principal ideal domain $D$ \cite[Theorem 4]{quillen 1976}.
Four years later, in \cite{Le Riche 1980},
Le Riche focused on investigating various ring-theoretic properties of this construction,
referring to it as $R\langle X\rangle$. Moreover, to the best of our knowledge,
it was eventually named the \textit{Serre conjecture's ring} in \cite{CEK97}.
The reader may refer to \cite{anderson 1985, BH80, H88, K89, Le Riche 1980, lucas 2020, WK16}
for more information on Nagata rings and Serre's conjecture rings.
Notably, in \cite[page 97]{anderson 1985}, the Anderson brothers and Markanda mentioned the quotient ring $R[X]_A$ (referred to as $R[X]_W$ in their work), and
they pointed out that $R[X]_W$ has applications in automata theory.
This construction has not received the same level of attention as the Nagata and Serre conjecture's rings.
Recently, Beak revisited this `new' construction in his Master's thesis in 2022 (referred to as $R[X]_B$ in his work) \cite{B22}.
Finally, in 2024,  Baek and Lim introduced the ring $R[X]_A$ using the concept of ‘localization at 0’, and referred to
it as the Anderson ring \cite{BL24}.

Now, it is convenient to review some properties of the three aforementioned constructions.
We first note that $R[X]_A$, $R[X]_N$ and $R[X]_U$ are faithfully flat $R$-modules,
meaning that they share many ring-theoretic properties with $R$.
More precisely, for any ideal $I$ of $R$, if $T$ denotes $R[X]_A$, $R[X]_U$ or $R[X]_N$,
then the equality $IT \cap R = I$ holds (\cite[Lemma 4.1]{BL24} and \cite[Theorem 2.2(1)]{anderson 1985}),
and $IT$ is finitely generated if and only if $I$ is finitely generated
(\cite[Lemma 3.5]{BL24} and \cite[Theorem 2.2(2)]{anderson 1985}).
Moreover, the maximal spectrum of $R[X]_A$ and $R[X]_N$ are as follows:
${\rm Max}(R[X]_A) = \{(M+XR[X])_A \,|\, M \in {\rm Max}(R)\}$ \cite[Theorem 2.1]{BL24} and
${\rm Max}(R[X]_N) = \{MR[X]_N \,|\, M \in {\rm Max}(R) \}$ \cite[Proposition 33.1(3)]{gilmer book}.
However, it is worth noting that there is no one-to-one correspondence between the maximal ideals of $R$ and those of $R[X]_U$ in general \cite[Lemma 3.2]{lucas 2020}.
Regarding Krull dimensions, we have $\dim(R[X]_A)=\dim(R[X])$ and
$\dim(R[X]_N)=\dim(R[X]_U)=\dim(R[X])-1$ for any finite dimensional ring $R$
(\cite[Proposition 2.4]{BL24} and \cite[Theorem 17.3 and Corollary 17.4]{H88}). 
In particular, as noted in \cite[Remark 2.5]{BL24}, when $R$ is a zero-dimensional ring,
the spectrum of $R[X]_A$ is given by
${\rm Spec}(R[X]_A) =\{MR[X]_A \,|\, M \in {\rm Max}(R) \} \cup \{(M+XR[X])_A \,|\, M \in {\rm Max}(R)\}$.

It is clear that $R[X]_A$ and $R[X]_U$ are subrings of $R[X]_N$
(because $A$ and $U$ are subsets of $N$).
We next consider the relation between $R[X]_A$ and $R[X]_U$.
Let $\widetilde{U}$ be the set of polynomials whose coefficient of the lowest degree term is $1$.
Then $\widetilde{U}$ is a multiplicative subset of $R[X]$ containing $A$,
so we obtain the quotient ring $R[X]_{\widetilde{U}}$ that contains $R[X]_A$.
More precisely, it is easy to check that $(R[X]_A)[\frac{1}{X}] = R[X]_{\widetilde{U}}$,
and that $R[X]_{\widetilde{U}}$ is isomorphic to $R[X]_U$.
This establishes the fact that $R[X]_A$ is a subring of $R[X]_U$ up to isomorphism.

In this paper, we continue investigating the properties of Anderson rings.
Specifically, we provide the properties of the Anderson ring $R[X]_A$ with its base ring $R$,
and we also establish correlations between the properties of Anderson ring and
those of both the Nagata rings and Serre's conjecture rings.

This paper consists of four section including introduction.
In Section \ref{Sec2},
we focus on Anderson rings over specific base rings and examine some conditions on $R$
for which $R[X]_A$ satisfies certain properties.
In particular, we characterize the associated prime ideals of $R[X]_A$, $R[X]_U$ and $R[X]_N$ and
we explore rings from the following classes:
$pm$-rings, $mp$-rings, and residually (P) rings.
We also provide some results regarding the ascending chain condition for Anderson rings.
In Section \ref{Sec3}, we devote our attention to the study of
Anderson rings of various Jaffard-like domains.
To do this, we first provide facts regarding the valuative dimension of the Anderson ring.
More precisely, we show that $\dim_v(D[X]_A) = 1+ \dim_v(D)$ (Lemma \ref{valuative dimension}).
After examining the above fact,
we show that $D$ is a Jaffard domain if and only if
$D[X]_A$ is a Jaffard domain with $\dim(D[X]) = 1+ \dim(D)$ (Theorem \ref{TheoJaffard}).
We then continue with the study of the Anderson rings of various Jaffard-like domain,
including locally Jaffard domains, maximally Jaffard domains, and MZ-Jaffard domains.
In Section \ref{Sec4},
we shift our focus to Anderson rings in relation to integral domains arising from the well-known star-operations $v$ and $t$.
First, to compare the prime $t$-ideals of $D[X]_A$ with those of $D$,
we investigate relationships between the $t$-dimension of $D$ and that of $D[X]_A$.
More precisely, we find a lower bound and an upper bound of the $t$-dimension of $D[X]_A$
in terms of the $t$-dimension of $D$ and that of $D[X]$ (Lemma \ref{t-Dim 0}).
In particular, we prove that if $t\mhyphen\dim(D[X]_A) = 1$, then $t\mhyphen\dim(D) = 1$, and
the converse holds when $D$ is a UMT-domain (Corollary \ref{t-Dim}).
Additionally, we examine the $t$-dimension of $D[X]_A$ for some specific integral domain.
After studying the $t$-dimension,
we show that if $D[X]_A$ is a $t$-SFT ring, then so is $D$,
and the converse holds when $D$ is integrally closed (Proposition \ref{t-SFT-ring}).
Moreover, using the obtained results about $t$-dimension of $D$ and that of $D[X]_A$,
we provide conditions on $D$ under which $D[X]_A$ becomes Krull-like domains
(e.g. weakly Krull domain, infra-Krull domain, generalized Krull domain,
and GK-domain) (Corollary \ref{Krull-like domains}).
Finally, we introduce the Anderson rings that are (locally) $v$-domains.

\section{Specific ring-theoretic properties of Anderson rings}\label{Sec2}

Let $R$ be a commutative ring with identity.
In this section, we investigate Anderson rings over specific rings,
and the conditions on $R$ when $R[X]_A$ satisfies particular conditions.

Before we start, we note some elementary results regarding the Anderson rings derived from the Nagata rings.

\begin{remark}\label{trivial results}
{\rm
Let $R$ be a commutative ring with identity and
let (P) be a property of $R$.
Consider the following conditions:
\begin{enumerate}
\item[(1)]
If $R$ satisfies property (P), then $R_S$ also satisfies property (P)
for any multiplicative subset $S$ of $R$.
\item[(2)]
If $R[X]_N$ satisfies property (P), then $R$ satisfies property (P).
\item[(3)]
If $R$ satisfies property (P), then $R[X]$ also satisfies property (P).
\end{enumerate}
If the conditions (1) and (2) hold,
then we obtain the following
\begin{itemize}
\item if $R[X]_A$ satisfies property (P), then so is $R$.
\end{itemize}
If, in addition, the condition (3) holds,
then the converse of the previous implication is also true.
}
\end{remark}

Let $R$ be a commutative ring with identity.
It is well-known that
every prime ideal is contained in a maximal ideal.
If such a maximal ideal is unique, the ring $R$ is called a 
{\it pm-ring}. Clearly, $R$ is a $pm$-ring if and only if 
the integral domain $R/P$ is a quasi-local ring
for all $P \in {\rm Spec}(R)$.
Typical examples of $pm$-rings include
zero-dimensional rings and quasi-local rings.
Note that an integral domain is a $pm$-ring if and only if
it is a quasi-local ring.
For instance, $\mathbb{Z}$ is not a $pm$-ring,
but the localization at $2 \mathbb{Z}$ is a $pm$-ring.

Subsequently, recall that an \textit{mp-ring} is a commutative ring with identity
in which every prime ideal contains a unique minimal prime ideal.
Clearly, zero-dimensional rings are $mp$-rings. Notably, for any field $K$,
the polynomial ring $K[X,Y]$ is an $mp$-ring,
whereas its factor ring $K[X,Y]/(XY)$ is not (see \cite[Remark 4.4]{AT20}).

Note that there is one-to-one correspondence between
the maximal ideals (respectively, minimal prime ideals) of $R$
and the maximal ideals (respectively, minimal prime ideals) of $R[X]_A$
given by $M \leftrightarrow (M+XR[X])_A$
(respectively, $P \leftrightarrow PR[X]_A$) \cite[Theorem 2.1]{BL24}.
Similarly, there is one-to-one correspondence between
the maximal ideals (respectively, minimal prime ideals) of $R$ and
the maximal ideals (respectively, minimal prime ideals) of $R[X]_N$
given by $M \leftrightarrow MR[X]_N$ (respectively, $P \leftrightarrow PR[X]_N$)
\cite[Theorem 14.1]{H88}.
From these facts, we obtain

\begin{theorem}\label{PM-ring}
Let $R$ be a commutative ring with identity.
The following assertions hold.
\begin{enumerate}
\item[(1)] The following assertions are equivalent.
\begin{enumerate}
\item[(i)]
$R$ is a $pm$-ring.
\item[(ii)]
$R[X]_A$ is a $pm$-ring.
\item[(iii)]
$R[X]_U$ is a $pm$-ring.
\item[(iv)]
$R[X]_N$ is a $pm$-ring.
\end{enumerate}
\item[(2)] If $R$ is not an integral domain, then the following assertions are equivalent.
\begin{enumerate}
\item[(i)]
$R$ is an $mp$-ring.
\item[(ii)]
$R[X]_A$ is an $mp$-ring.
\item[(iii)]
$R[X]_U$ is an $mp$-ring.
\item[(iv)]
$R[X]_N$ is an $mp$-ring.
\end{enumerate}
\end{enumerate}
\end{theorem}

\begin{proof}
(1) The equivalences (i) $\Leftrightarrow$ (iii) $\Leftrightarrow$ (iv)
is already proven in \cite[Lemma 2.2]{KB16}.
Hence we only consider the implications (i) $\Rightarrow$ (ii) $\Rightarrow$ (iv).

(i) $\Rightarrow$ (ii)
Suppose that $R$ is a $pm$-ring and
let $\mathfrak{p}$ be a prime ideal of $R[X]_A$.
Suppose to the contrary that $\mathfrak{p}$ is contained in
two distinct maximal ideals $(M_1+XR[X])_A$ and $(M_2+XR[X])_A$
of $R[X]_A$, where $M_1,M_2 \in {\rm Max}(R)$.
Then $\mathfrak{p} \cap R$ is contained in $M_1$ and $M_2$.
Since $\mathfrak{p} \cap R$ is a prime ideal of $R$ and
$R$ is a $pm$-ring,
this is a contradiction.
Hence $\mathfrak{p}$ is contained in only one maximal ideal of $R[X]_A$.
Thus $R[X]_A$ is a $pm$-ring.

(ii) $\Rightarrow$ (iv)
Suppose that $R[X]_A$ is a $pm$-ring.
Let $\mathfrak{p}$ be a prime ideal of $R[X]_N$.
Suppose to the contrary that $\mathfrak{p}$
is contained in two distinct maximal ideal $M_1R[X]_N$ and $M_2R[X]_N$ of $R[X]_N$,
where $M_1,M_2 \in {\rm Max}(R)$.
Then $\mathfrak{p} \cap R[X]_A$ is contained in $M_1R[X]_A$ and $M_2R[X]_A$,
so it is contained in $(M_1+XR[X])_A$ and $(M_2+XR[X])_A$.
This contradicts the fact that $R[X]_A$ is a $pm$-ring.
Hence $\mathfrak{p}$ is contained in only one maximal ideal of $R[X]_A$.
Thus $R[X]_N$ is a $pm$-ring.

(2) This result holds in the same way as (1) by using \cite[Theorem 2.1(1)]{BL24}.
\end{proof}

Let $R$ be a commutative ring with identity.
Recall that $R$ is of {\it finite character} if
every nonzero nonunit element is contained in only finitely many maximal ideals of $R$;
or equivalently, every proper ideal of $R$ is contained in finitely many maximal ideals of $R$.
For instance, semi-quasi-local rings have finite character.
When $R$ is a $pm$-ring of finite character,
it is called an {\it h-local ring}.
The following result is a direct consequence of
Theorem \ref{PM-ring}(1), \cite[Proposition 2.3]{BL24}, and \cite[Lemma 8(1)]{Li15}.

\begin{corollary}
Let $R$ be a commutative ring with identity.
Then the following assertions hold.
\begin{enumerate}
\item[(1)]
$R$ is both a semi-quasi-local ring and a $pm$-ring if and only if
$R[X]_A$ is an h-local ring.
\item[(2)]
$R$ is an h-local ring if and only if
$R[X]_N$ is an h-local ring.
\end{enumerate}
\end{corollary}

A commutative ring $R$ with identity is called a \textit{PF-ring} if
all of its principal ideals are flat $R$-modules. Note that every 
PF-ring is a reduced ring. More precisely, PF-rings coincide with 
reduced $mp$-rings, as shown in \cite[Theorem 6.4]{AT20}.
Hence we obtain the next result follows directly from Theorem \ref{PM-ring}(2).

\begin{corollary}
Let $R$ be a commutative ring with identity. 
The following statements are equivalent:
\begin{enumerate}
\item[(1)] 
$R$ is a PF-ring.
\item[(2)] 
$R[X]$ is a PF-ring.
\item[(3)] 
$R[X]_A$ is a PF-ring.
\item[(4)] 
$R[X]_U$ is a PF-ring.
\item[(5)]
$R[X]_N$ is a PF-ring.
\end{enumerate}
\end{corollary}

Let $R$ be a commutative ring with identity.
For a property (P) of integral domains,
$R$ is said to be {\it residually} (P)
if $R/P$ has the property (P)
for all $P \in {\rm Spec}(R)$.
Now, we investigate the condition of $R$ under which
$R[X]_A$ becomes some specific residually (P) ring.
To do this, we need the fact about the factor ring of the Anderson rings.

Throughout this paper,
for an ideal $I$ of $R$ and
a multiplicative subset $S$ of $R$ disjoint from $S$,
we denote $(S+I)/I = \{s + I \,|\, s \in S\}$.

\begin{lemma}{\rm (\cite[Proposition 3.1(3)]{B22})}\label{LemResi2}
Let $R$ be a commutative ring with identity.
For any ideal $I$ of $R$,
the ring $R[X]_A/IR[X]_A$ is isomorphic to $(R/I)[X]_{A/IR[X]}$.
\end{lemma}

\begin{proof}
Note that for any ideal $I$ of $R$ and
a multiplicative subset $S$ of $R$ disjoint from $I$,
$R_S/IR_S$ is isomorphic to $(R/I)_{(S+I)/I}$
\cite[Proposition 5.8]{gilmer book}.
Since $IR[X]$ is disjoint from $A$ for all ideals $I$ of $R$, the proof is done.
\end{proof}

\begin{remark}\label{trivial residually}
{\rm
Let $D$ be an integral domain and
let (P) denote the property of integral domains
which satisfies the conditions (1) and (2) in Remark \ref{trivial results}.
Then we obtain the fact that
if $D[X]_A$ has the property (P), then $D$ has the property (P).

Now, let $R$ be a commutative ring with identity and
suppose that $R[X]_A$ is residually (P) and
let $P$ be a prime ideal of $R$.
Then $PR[X]_A$ is a prime ideal of $R[X]_A$,
so $R[X]_A/PR[X]_A$ has the property (P).
This implies that $(R/P)[X]_{A/PR[X]}$
has the property $({\rm P})$ by Lemma \ref{LemResi2}.
Hence $R/P$ also has the property (P).
Thus $R$ is residually (P).
Consequently, if $R[X]_A$ is a residually (P) ring,
then so is $R$.
}
\end{remark}

Now, we investigate the case when (P) is Noetherian.
To do this, we need the following lemma.

\begin{lemma}
Let $R$ be a commutative ring with identity.
Then $R$ is a residually Noetherian ring if and only if 
$R/P$ is a Noetherian domain for every minimal prime ideal $P$ of $R$.
\end{lemma}

\begin{proof}
The direct implication is obvious.
For the converse, assume that $R/P$ is Noetherian for any minimal prime ideal $P$ of $R$.
Let $Q$ be a non minimal prime ideal of $R$.
Then there is a minimal prime ideal $P$ of $R$ contained in $Q$.
Hence $R/P$ is Noetherian, and
thus $R/Q$ is also Noetherian since $R/Q\cong (R/P)/(Q/P)$.
\end{proof}

By the previous lemma, we obtain

\begin{proposition}
Let $R$ be a commutative ring with identity. 
Then $R[X]_A$ is a residually Noetherian ring
if and only if $R$ is a residually Noetherian ring.
\end{proposition}

\begin{proof}
If $R$ is an integral domain, then the result is obvious since
an integral domain which is a residually Noetherian ring is exactly a Noetherian domain, and
note that the Noetherian domains satisfy the conditions (1), (2) and (3) in Remark \ref{trivial results}.
Hence we may assume that $R$ is not an integral domain.
The direct implication follows directly from Remark \ref{trivial residually} by using the fact that
$D$ is a Noetherian domain if and only if $D[X]_A$ is a Noetherian domain.
For the converse, suppose that $R$ is a residually Noetherian ring and 
let $Q$ be a minimal prime ideal of $R[X]_A$. Then $Q=PR[X]_A$
for some minimal prime ideal $P$ of $R$ by 
\cite[Theorem 2.1(1)]{BL24}. Since $R/P$ is a 
Noetherian domain, it follows from Lemma~\ref{LemResi2} 
that 
$R[X]_A/Q= R[X]_A/PR[X]_A \cong (R/P)[X]_{A/PR[X]}$ 
is also a Noetherian domain. 
Therefore $R[X]_A$ is a residually Noetherian ring, 
and the proof is complete.
\end{proof}

Let $D$ be an integral domain.
Recall that
\begin{itemize}
\item $D$ is a {\it valuation domain} if
either $(a) \subseteq (b)$ or $(b) \subseteq (a)$ for any
$a,b \in D$,
\item $D$ is a {\it B\'ezout domain} if
every finitely generated ideal is principal,
\item $D$ is a {\it Pr\"ufer domain} if
every finitely generated ideal of $D$ is invertible.
\end{itemize}
Now, we examine the case where (P) is such integral domains.

\begin{theorem}
Let $R$ be a commutative ring with identity.
Then the following assertions are equivalent.
\begin{enumerate}
\item[(1)]
$R[X]_A$ is a residually principal ideal ring.
\item[(2)]
$R[X]_A$ is a residually valuation ring.
\item[(3)]
$R[X]_A$ is a residually B\'ezout ring.
\item[(4)]
$R[X]_A$ is a residually Pr\"ufer ring.
\item[(5)]
$R$ is a zero-dimensional ring.
\end{enumerate}
\end{theorem}

\begin{proof}
The implications (1) $\Rightarrow$ (3) $\Rightarrow$ (4) and
(2) $\Rightarrow$ (3) $\Rightarrow$ (4) are obvious.

(4) $\Rightarrow$ (5)
Suppose that $R[X]_A$ is a residually Pr\"ufer ring.
Let $P \in {\rm Spec}(R)$.
Then $R[X]_A/PR[X]_A$ is a Pr\"ufer domain.
By Lemma \ref{LemResi2},
$(R/P)[X]_{A/PR[X]}$ is a Pr\"ufer domain.
Hence $R/P$ is a field \cite[Corollary 3.11]{BL24},
so $P$ is a maximal ideal.
Thus $R$ is a zero-dimensional ring.

Now, suppose that $R$ is zero-dimensional.
Then ${\rm Spec}(R[X]_A) = \{MR[X]_A \,|\, M \in {\rm Max}(R)\} \cup
\{(M+XR[X])_A \,|\, M \in {\rm Max}(R)\}$ \cite[Remark 2.5(2)]{BL24}.
Hence we only consider the ideals which are of the form $MR[X]_A$.
Note that $(R/M)[X]_{A/MR[X]}$ is both a PID and a valuation domain \cite[Corollary 3.11]{BL24}.
Hence $R[X]_A/MR[X]_A$ is both a PID and a valuation domain by Lemma \ref{LemResi2}.
Thus $R[X]_A$ is both a residually principal ideal ring and a residually valuation ring.
Consequently, the implications (5) $\Rightarrow$ (1) and
(5) $\Rightarrow$ (2) hold.
\end{proof}

At the end of this section,
we investigate specific ascending chain conditions for the Anderson rings.
Let $R$ be a commutative ring with identity and
let $n$ be a positive integer.
An ideal $I$ of $R$ is an {\it $n$-generated ideal} if
$I$ is generated by $n$ elements of $R$.
According to \cite{HL83},
$R$ has {\it $n$-ascending chain condition} if it satisfies the ascending 
chain condition on $n$-generated ideals. 
Notice that the $n$-ascending chain condition property does not generally
hold well under localization or passage to a polynomial ring 
(see examples of \cite[Section 4]{HL83}).

\begin{proposition}\label{nACC}
Let $R$ be a commutative ring with identity and
let $n$ be a fixed positive integer.
If $R[X]_A$ has $n$-ascending chain condition, then so is $R$.
\end{proposition}

\begin{proof}
Assume that $R[X]_A$ has $n$-ascending chain condition and
let $I_1 \subseteq I_2 \subseteq \cdots$ 
be a chain of $n$-generated ideals of $R$.
Then we have a chain
$I_1R[X]_A \subseteq I_2R[X]_A \subseteq \cdots$
of $n$-generated ideals of $R[X]_A$ \cite[Lemma 3.5]{BL24}.
Since $R[X]_A$ has $n$-ascending chain condition, 
there exists a positive integer $m \in \mathbb{N}$ such that
$I_mR[X]_A = I_{\ell}R[X]_A$ for all $\ell \geq m$.
Hence $I_m = I_{\ell}$ for all $\ell \geq m$ \cite[Lemma 4.1]{BL24}.
Thus $R$ has $n$-ascending chain condition.
\end{proof}

As a consequence of the previous result, we obtain

\begin{corollary}{\rm (\cite[Proposition 4.4]{B22})}\label{ACCP}
Let $R$ be a commutative ring with identity.
If $R[X]_A$ satisfies ascending chain condition on principal ideals, then so is $R$.
\end{corollary}

\begin{proposition}
Let $R$ be a commutative ring with identity and 
let $({\rm P})$ denote the property of integral domains
satisfying the ascending chain condition on principal ideals.
If $R[X]_A$ is a residually $({\rm P})$ ring, then so is $R$.
\end{proposition}

\begin{proof}
Assume that $R[X]_A$ is a residually $({\rm P})$ ring and let $P$ be a 
prime ideal of $R$. Then $PR[X]_A$ is a prime ideal of $R[X]_A$, 
and hence $R[X]_A/PR[X]_A$ satisfies ascending chain condition on principal ideals. 
By Lemma \ref{LemResi2}, $(R/P)[X]_{A/PR[X]}$
also satisfies ascending chain condition on principal ideals.
Therefore $R/P$ inherits this property according to Corollary \ref{ACCP}.
Consequently, $R$ is a residually (P) ring.
\end{proof}

From \cite{BH74}, we recall that a prime ideal $P$ of $D$ is called
an {\it associated prime of a principal ideal} $aD$ of $D$
if $P$ is minimal over $(aD:bD)$ for some $b\in D\setminus aD$.
For the sake of brevity, we call $P$ an \textit{associated prime ideal} of $D$,
denoted by $P \in {\rm Ass}(D)$,
if $P$ is an associated prime ideal of some principal ideal of $D$.
By the next result,
we can characterize the associated prime ideals of $D[X]_A$.

\begin{lemma}\label{PrAssPri}
Let $D$ be an integral domain.
Then the following two types are the only associated prime ideals of $D[X]_A$.
\begin{enumerate}
\item[(1)]
$PD[X]_A$ where $P \in {\rm Ass}(D)$, or
\item[(2)]
$\mathfrak{p}D[X]_A$, where $\mathfrak{p}$ is an upper to zero in $D[X]$ disjoint from $A$.
\end{enumerate}
\end{lemma}

\begin{proof}
Let $P$ be an associated prime ideal of $D$ and
let $\mathfrak{p}$ be an upper to zero in $D[X]$ disjoint from $A$.
Then $PD[X]_A$ and $\mathfrak{p}D[X]_A$ are  associated prime ideals of $D[X]_A$ by \cite[Lemma 1 and Corollary 8]{BH74}.
Let $\mathfrak{q}$ be a associated prime ideal of $D[X]_A$.
Then there exists $\frac{f_1}{g_1}, \frac{f_2}{g_2} \in D[X]_A$ such that
$\mathfrak{q}$ is minimal over $(\frac{f_1}{g_1} : \frac{f_2}{g_2})D[X]_A
= (f_1 : f_2)D[X]_A$.
Hence $\mathfrak{q} \cap D[X]$ is minimal over $(f_1:f_2)$,
which means that $\mathfrak{q} \cap D[X]$ is an associated prime ideal of $D[X]$.
Hence we obtain either $\mathfrak{q} \cap D[X] = PD[X]$ or $\mathfrak{q} \cap D[X] = \mathfrak{p}$,
where $P \in {\rm Ass}(D)$ and
$\mathfrak{p}$ is an upper to zero disjoint from $A$
\cite[Corollary 8]{BH74}.
Thus the proof is complete.
\end{proof}

Note that a similar result to the previous one can be obtained for $D[X]_N$ and $D[X]_U$.

Let $D$ be an integral domain. From \cite{KT22}, 
we recall that $D$ satisfies the ascending chain condition on associated 
prime ideals if any ascending chain of associated 
prime ideals of $D$ is stationary.
Notice that one-dimensional integral domains and (locally) 
Noetherian domains satisfy the ascending chain condition on associated prime ideals,
but the converse does not hold in general \cite[Remark 4.2]{KT22}.

\begin{proposition}
Let $D$ be an integral domain.
Then $D[X]_A$ satisfies the ascending chain condition on associated prime ideals if and only if
$D$ satisfies the ascending chain condition on associated prime ideals.
\end{proposition}

\begin{proof}
We can prove the direct implication in a similar way to the proof of Proposition \ref{nACC},
using Lemma \ref{PrAssPri}.
For the converse,
suppose that $\mathcal{C}: \mathfrak{p}_1 \subseteq \mathfrak{p}_2 \subseteq \cdots$ is an ascending chain of
associated prime ideals of $D[X]_A$.
Then for each $i \geq 3$,
there exists an associated prime ideal $P_i$ of $D$ such that $\mathfrak{p}_i = P_iD[X]_A$ by Lemma \ref{PrAssPri}.
Hence we obtain the chain $P_3 \subseteq P_4 \subseteq \cdots$ of
associated prime ideals of $D$.
By our assumption, there exists a positive integer $n \geq 3$ such that
$P_n = P_{\ell}$ for all $\ell \geq n$.
This means that the chain $\mathcal{C}$ is stationary.
Thus $D[X]_A$ satisfies the ascending chain condition on associated prime ideals.
\end{proof}

\section{Jaffard-like domains arising from Anderson rings}\label{Sec3}

Let $D$ be an integral domain.
Recall that the {\it valuative dimension} of $D$, denoted by $\dim_v(D)$,
is defined to be the supremum of the Krull dimensions 
of the valuation overrings of $D$.
It is easy to show that
$\dim(D) \leq \dim_v(D)$ and $\dim_v(T) \leq \dim_v(D)$
for any overring $T$ of $D$.
Also, recall that $D$ is a {\it Jaffard domain} if
$\dim_v(D)$ is finite and $\dim_v(D) = \dim(D)$.
For instance, finite dimensional Noetherian domains and Pr\"ufer domains
are Jaffard domains.

In this section, we examine the conditions on $D$ under which
the Anderson rings become Jaffard-like domains.
To do this, we begin by evaluating the valuative dimension of the Anderson rings 
in terms of the valuative dimension of the base integral domain.
In particular, we derive an analogue of \cite[Proposition 2.4]{BL24}
for valuative dimensions.

\begin{lemma}\label{valuative dimension}
Let $D$ be an integral domain.
Then $\dim_v(D[X]_A)=1+\dim_v(D)$, and hence
$\dim_v(D[X]_A) = \dim_v(D[X])$.
Moreover, $D[X]$ is a Jaffard domain if and only if $D[X]_A$ is a Jaffard domain.
\end{lemma}

\begin{proof}
Note that $\dim_v(D) = \dim_v(D[X]_N) \leq \dim_v(D[X]_A)$,
where the equality follows directly from \cite[Corollary 1.4]{CEK97},
so the desired equality holds when $\dim_v(D)$ is infinite.
Hence we may assume that $\dim_v(D)$ is finite. 
Let $V$ be a valuation overring of $D$ such that $\dim_v(D)=\dim(V)$.
We first claim that $\dim_v(V[X]_{A_V})=1 + \dim(V)$,
where $A_V = \{f \in V[X] \,|\, f(0) = 1\}$.
As a matter of fact,
since $V[X]_{A_V}$ is an overring of $V[X]$,
we have $\dim_v(V[X]_{A_V}) \leq \dim_v(V[X])$.
Since $V$ is a Jaffard domain,
we obtain that $\dim_v(V[X])=\dim(V[X])=1+\dim(V)$ \cite[Corollary 30.12]{gilmer book}.
This implies that $\dim_v(V[X]_{A_V}) \leq 1+\dim(V)$.
In addition, we have $\dim(V[X]) = \dim(V[X]_{A_V}) \leq \dim_v(V[X]_{A_V})$,
where the first equality follows directly from 
\cite[Proposition 2.4]{BL24}.
Hence $1+\dim(V) \leq \dim_v(V[X]_{A_V})$.
Therefore $\dim_v(V[X]_{A_V})=1+\dim(V)$.
As $D[X]_A$ lies between $D[X]$ and $V[X]_{A_V}$,
we have $\dim_v(V[X]_{A_V})\leq \dim_v(D[X]_A)\leq \dim_v(D[X])$.
Thus $\dim_v(D[X]_A)=1+\dim_v(D)$ since $\dim_v(D[X]) = 1 + \dim_v(D)$
\cite[Th\'eor\`eme 2, page 60]{J60}.
The remainder argument follows directly from
\cite[Th\'eor\`eme 2, page 60]{J60}.

For the remaining argument,
suppose that $D[X]$ is a Jaffard domain.
Then $\dim_v(D[X]_A) = \dim_v(D[X]) = \dim(D[X]) = \dim(D[X]_A)$.
Hence $D[X]_A$ is a Jaffard domain.
Similarly, the converse holds.
\end{proof}

Note that, as mentioned in \cite[Remark 1.3(c)]{ABDFK88},
$D$ is a Jaffard domain if and only if
$D[X]$ is a Jaffard domain with Krull dimension equal to $1+\dim(D)$.
In what follows, we provide a similar result for Anderson rings.

\begin{theorem}\label{TheoJaffard}
Let $D$ be an integral domain.
Then the following assertions are equivalent.
\begin{enumerate}
\item[(1)]
$D$ is a Jaffard domain.
\item[(2)]
$D[X]$ is a Jaffard domain with $\dim(D[X]) = 1+\dim(D)$.
\item[(3)]
$D[X]_A$ is a Jaffard domain with $\dim(D[X]) = 1+\dim(D)$.
\item[(4)]
$D[X]_U$ is a Jaffard domain with $\dim(D[X]) = 1+\dim(D)$.
\item[(5)]
$D[X]_N$ is a Jaffard domain with $\dim(D[X]) = 1+\dim(D)$.
\end{enumerate}
\end{theorem}

\begin{proof}
The implication (1) $\Rightarrow$ (2)
is already proved in \cite[Remark 1.3(c)]{ABDFK88},
and the equivalences (2) $\Leftrightarrow$ (4) $\Leftrightarrow$ (5)
follow directly from \cite[Proposition 2.2]{CEK97}.
Also, the implication (2) $\Rightarrow$ (3) follows directly from Lemma \ref{valuative dimension}.
Hence it is enough to show the implication (3) $\Rightarrow$ (1).

 (3) $\Rightarrow$ (1)
Suppose that $D[X]_A$ is a Jaffard domain with $\dim(D[X]) = 1+\dim(D)$.
Then from Lemma \ref{valuative dimension} we obtain
\begin{eqnarray*}
\dim_v(D) &=& \dim_v(D[X]_A)-1\\
&=&\dim(D[X]_A)-1\\
&=&\dim(D[X])-1\\
&=&\dim(D).
\end{eqnarray*}
Thus $D$ is a Jaffard domain.   
\end{proof}

The following example shows that
we can construct a Jaffard domain which neither Pr\"ufer domain nor (locally) Noetherian domain issued from Anderson rings.

\begin{example}
{\rm
Consider $D = \mathbb{Z} + Y \mathbb{Q}[Y]$. It is known that
$D$ is a two-dimensional Pr\"ufer domain, and
then it is a Jaffard domain.
By Theorem \ref{TheoJaffard},
we obtain $D[X]_A$ is a Jaffard domain of Krull dimension three.
However, since $D$ is not a field,
$D[X]_A$ is not a Pr\"ufer domain \cite[Corollary 3.11]{BL24}.
In addition, since $D$ is not locally Noetherian,
it follows from \cite[Proposition 2.7]{BL24} that $D[X]_A$ is also not locally Noetherian
(recall the fact that locally Noetherian Pr\"ufer domains must be one-dimensional).

More generally, suppose that $D$ is a Pr\"ufer domain which is not a field.
Then $D[X]_A$ is a Jaffard domain that is not a Pr\"ufer domain
by using Theorem \ref{TheoJaffard} and \cite[Corollary 3.11]{BL24}, respectively.
}
\end{example}

\begin{remark}\label{RemJaffard}
{\rm
According to \cite[Proposition 1.2(a)]{ABDFK88},
if $D$ is a Jaffard domain, then $ D[X]$ is also a Jaffard domain.
However, the integral domain $D=k+Yk(Z)[Y]_{(Y)}$,
where $Y$ and $Z$ are two indeterminates over a field $k$,
is known to be a one-dimensional non Jaffard domain but $D[X]$ is a Jaffard domain
with $\dim(D[X])=3$ \cite[Example 3.6]{FK90}.
In this case, it follows from Lemma \ref{valuative dimension} and
Theorem \ref{TheoJaffard} that $D[X]_A$ is a Jaffard domain with $\dim(D[X]) \neq 1 + \dim(D)$.
}
\end{remark}

Let $D$ be an integral domain.
We say that $D$ is a {\it maximally Jaffard domain} if
$D_M$ is a Jaffard domain for any $M \in {\rm Max}(D)$.
It is known that for any integral domain $D$,
we have $\dim_v(D) = \sup\{\dim_v(D_M) \mid M \in \mathrm{Max}(D)\}$
\cite[Theorem 5.4.13]{WK16}.
From this, it easily follows that a maximally Jaffard domain
with finite valuative dimension must also be a Jaffard domain.
However, the integral domain $D$ constructed in Example 2.3 of \cite{FK04}
is Jaffard but not maximally Jaffard.

\begin{corollary}\label{maximally Jaffard 1}
Let $D$ be an integral domain with finite valuative dimension.
If $D$ is a maximally Jaffard domain,
then $D[X]_A$ is a maximally Jaffard domain with $\dim(D[X])=1+\dim(D)$.
\end{corollary}

\begin{proof}
Assume that $D$ is a maximally Jaffard domain.
Then $D$ is a Jaffard domain, and hence $\dim(D[X])=1+\dim(D)$
\cite[Remark 1.3(c)]{ABDFK88}.
Now, let $\mathfrak{m}$ be a maximal ideal of $ D[X]_A$.
We have that $\mathfrak{m}$ is of the form $(M+XD[X])_A$
for some maximal ideal $M$ of $D$ \cite[Theorem 2.1(2)]{BL24}.
Since $D_M$ is a Jaffard domain,
$D_M[X]_{A_M}$ is also a Jaffard domain by Theorem \ref{TheoJaffard}.
Thus $D[X]_A$ is a maximally Jaffard domain.
\end{proof}

Let $D$ be an integral domain and
let $S$ be a multiplicative subset of $D$.
Note that $(D_S)_{MD_S} = D_M$ for any maximal ideal $M$ of $D$. 
Hence we directly obtain the following

\begin{lemma}\label{maximally Jaffard lemma}
Let $D$ be an integral domain.
Then $D$ is a maximally Jaffard domain if and only if
$D_S$ is a maximally Jaffard domain for each
multiplicative subset $S$ of $D$.
\end{lemma}

Now, we explore the converse of Corollary \ref{maximally Jaffard lemma}.
By Theorem \ref{TheoJaffard} and Lemma \ref{maximally Jaffard lemma},
we can easily obtain

\begin{proposition}
Let $D$ be an integral domain.
Then the following assertions are equivalent.
\begin{enumerate}
\item[(1)]
$D$ is a maximally Jaffard domain.
\item[(2)]
$D[X]_A$ is a maximally Jaffard domain with
$\dim(D_M[X]) = 1+\dim(D_M)$ for any $M \in {\rm Max}(D)$.
\item[(3)]
$D[X]_U$ is a maximally Jaffard domain with
$\dim(D_M[X]) = 1+ \dim(D_M)$ for any $M \in {\rm Max}(D)$.
\item[(4)]
$D[X]_N$ is a maximally Jaffard domain with
$\dim(D_M[X]) = 1+ \dim(D_M)$ for any $M \in {\rm Max}(D)$.
\end{enumerate}
\end{proposition}

\begin{proof}
(1) $\Rightarrow$ (2)
Suppose that $D$ is a maximally Jaffard domain.
Then $D_M$ is a Jaffard domain for any $M \in {\rm Max}(D)$.
Hence by Theorem \ref{TheoJaffard},
we obtain $D_M[X]_{A_M}$ is a Jaffard domain with
$\dim(D_M[X]) = 1+\dim(D_M)$.
Thus the assertion (2) holds.

The implications (2) $\Rightarrow$ (3) $\Rightarrow$ (4)
follows directly from Lemma \ref{maximally Jaffard lemma}.
 
(4) $\Rightarrow$ (1)
Assume that the assertion (4) holds.
Let $M$ be a maximal ideal of $D$.
Then $D_M[X]_{N_M}$ is a Jaffard domain,
so $D_M$ is a Jaffard domain by Theorem \ref{TheoJaffard}.
Thus $D$ is a maximally Jaffard domain.   
\end{proof}

We now turn our attention to locally Jaffard domains.
According to \cite{ABDFK88}, an integral domain $D$ 
is said to be a {\it locally Jaffard domain} if
$D_P$ is a Jaffard domain for any $P \in {\rm Spec}(D)$. 
It is important to note that a locally Jaffard domain 
with finite valuative dimension is a Jaffard 
domain \cite[Proposition 1.5(b)]{ABDFK88}.
First, we show that $D[X]$, $D[X]_A$, $D[X]_U$ and $D[X]_N$ are simultaneously locally Jaffard domains
for integral domains $D$ with finite Krull dimension.

\begin{lemma}\label{LemLocJaffard}
Let $D$ be an integral domain with finite $($Krull$)$ dimension.
Then the following assertions are equivalent.
\begin{enumerate}
\item[(1)] 
$D[X]$ is a locally Jaffard domain.
\item[(2)]
$D[X]_A$ is a locally Jaffard domain.
\item[(3)]
$D[X]_U$ is a locally Jaffard domain.
\item[(4)]
$D[X]_N$ is a locally Jaffard domain.
\end{enumerate}
\end{lemma}

\begin{proof}
The equivalences (1) $\Leftrightarrow$ (3) $\Leftrightarrow$ (4)
are immediate from \cite[Proposition 2.3]{CEK97}.

The implications (1) $\Rightarrow$ (2) $\Rightarrow$ (3)
follow directly from the fact that
$D[X]_N$ is the quotient ring of $D[X]_A$ by $N$,
and the locally Jaffard property is preserved under
any quotient ring of itself.
\end{proof}

\begin{proposition}\label{PropoLocallyJaffard}
Let $D$ be an integral domain with finite $($Krull$)$ dimension.
Then the following assertions are equivalent.
\begin{enumerate}
\item[(1)]
$D$ is a locally Jaffard domain.
\item[(2)]
$D[X]$ is a locally Jaffard domain with 
$\dim(D_P[X]) = 1 +\dim(D_P) $ for every $P \in \mathrm{Spec}(D)$.
\item[(3)]
$D[X]_A$ is a locally Jaffard domain with 
$\dim(D_P[X]) = 1 +\dim(D_P) $ for every $P \in \mathrm{Spec}(D)$.
\item[(4)]
$D[X]_U$ is a locally Jaffard domain with 
$\dim(D_P[X]) = 1 + \dim(D_P) $ for every $P \in \mathrm{Spec}(D)$.
\item[(5)]
$D[X]_N$ is a locally Jaffard domain with 
$\dim(D_P[X]) = 1 + \dim(D_P) $ for every $P \in \mathrm{Spec}(D)$.
\end{enumerate}
\end{proposition}

\begin{proof}
By applying Lemma \ref{LemLocJaffard}, it suffices to prove
the implications (1) $\Rightarrow$ (2) and (5) $\Rightarrow$ (1).
    
(1) $\Rightarrow$ (2)
Assume that $D$ is a locally Jaffard domain. 
Then $D[X]$ is a locally Jaffard domain \cite[Proposition 1(i)]{Ca90}.
Also, since $D_P$ is a Jaffard domain for any $P \in {\rm Spec}(D)$,
we obtain that $\dim(D_P[X]) = 1 + \dim(D_P)$
for any $P \in {\rm Spec}(D)$ by Theorem \ref{TheoJaffard}.
 
(5) $\Rightarrow$ (1)
Assume that $D[X]_N$ is a locally Jaffard domain
with $\dim(D_P[X]) = \dim(D_P)+1$ for every $P \in {\rm Spec}(D)$. 
Let $P$ be a prime ideal of $D$.
Then $PD[X]_N$ is a prime ideal of $D[X]_N$,
so $(D[X]_N)_{PD[X]_N}=D_P[X]_{N_P}$ is a Jaffard domain.
Since $\dim(D_P[X]) = 1 + \dim(D_P)$, 
it follows from Theorem \ref{TheoJaffard} that
$D_P$ is a Jaffard domain.
Thus $D$ is a locally Jaffard domain.
\end{proof}

In \cite{KT22}, the authors defined an integral domain $D$ to be a
{\it Mott-Zafrullah Jaffard domain} (in short, an MZ-Jaffard domain)
if $D_P$ is a Jaffard domain for any $P\in\mathrm{Ass}(D)$.
In the finite (Krull) dimensional setting,
it is clear that both locally essential domains and ($t$-)locally Jaffard domains 
are subclasses of $\mathrm{MZ}$-Jaffard domains.  
It is easy to check that in the case of one-dimensional integral domains,
the concepts of Jaffard domains, maximally Jaffard domains, 
locally Jaffard domains, and $\mathrm{MZ}$-Jaffard domains 
are coincident. As mentioned in \cite[Remark 2.6]{KT22}, it is 
possible for $D[X]$ to be an $\mathrm{MZ}$-Jaffard domain 
even when $D$ is not. Furthermore,  \cite[Example 2.3]{FK04} 
provides an example of a Jaffard domain that is not 
an $\mathrm{MZ}$-Jaffard domain.
At the end of this section,
we investigate the conditions on $D$ or $D[X]$ under which
$D[X]_A$ becomes an MZ-Jaffard domain.

\begin{lemma}\label{LemMZJaffard}
Let $D$ be an integral domain.
Then the following assertions are equivalent.
\begin{enumerate}
\item[(1)] 
$D[X]$ is an MZ-Jaffard domain.
\item[(2)]
$D[X]_A$ is an MZ-Jaffard domain.
\item[(3)]
$D[X]_U$ is an MZ-Jaffard domain.
\item[(4)]
$D[X]_N$ is an MZ-Jaffard domain.
\end{enumerate}
\end{lemma}

\begin{proof}
The implications (1) $\Rightarrow$ (2) $\Rightarrow$ (3) $\Rightarrow$ (4)
follow from \cite[Proposition 1.4]{KT22}.

(4) $\Rightarrow$ (1)
Assume that $D[X]_N$ is an MZ-Jaffard domain and
let $\mathfrak{p}$ be an associated prime ideal of $D[X]$. 
Set $P=\mathfrak{p}\cap D$.
If $P = (0)$, then $D[X]_\mathfrak{p}$ is a quasi-local PID,
and hence it is a Jaffard domain.
Now, assume that $P \neq (0)$.
By \cite[Corollary 8]{BH74}, we have $\mathfrak{p}=PD[X]$,
and then $D[X]_\mathfrak{p}=(D[X]_N)_{PD[X]_N}$ is a Jaffard domain
since $PD[X]_N$ is an associated prime ideal of $D[X]_N$
\cite[Lemma 2.2(b)]{KT22}.
Thus $D[X]$ is an $\mathrm{MZ}$-Jaffard domain.
\end{proof}

\begin{proposition}\label{PropoMZJaffard}
Let $D$ be an integral domain.
Then the following assertions are equivalent.
\begin{enumerate}
\item[(1)]
$D$ is an MZ-Jaffard domain.
\item[(2)]
$D[X]$ is an MZ-Jaffard domain with 
$\dim(D_P[X]) = 1 + \dim(D_P)$ for every $P \in \mathrm{Ass}(D)$.
\item[(3)]
$D[X]_A$ is an MZ-Jaffard domain with 
$\dim(D_P[X]) = 1 + \dim(D_P)$ for every $P \in \mathrm{Ass}(D)$.
\item[(4)]
$D[X]_U$ is an MZ-Jaffard domain with 
$\dim(D_P[X]) = 1 + \dim(D_P)$ for every $P \in \mathrm{Ass}(D)$.
\item[(5)]
$D[X]_N$ is an MZ-Jaffard domain with 
$\dim(D_P[X]) = 1 + \dim(D_P)$ for every $P \in \mathrm{Ass}(D)$.
\end{enumerate}
\end{proposition}

\begin{proof}
By Lemma \ref{LemMZJaffard}, we only need to show 
(1) $\Rightarrow$ (2) and (5) $\Rightarrow$ (1).
    
(1) $\Rightarrow$ (2) Assume that $D$ is an $\mathrm{MZ}$-Jaffard domain.
Then $D[X]$ is an $\mathrm{MZ}$-Jaffard domain
by \cite[Proposition 2.5]{KT22}.
On the other hand, since $D_P$ is a Jaffard domain for any $P \in {\rm Ass}(D)$,
the equality $\dim(D_P[X]) = 1 + \dim(D_P)$ holds by Theorem \ref{TheoJaffard}.

The implication (5) $\Rightarrow$ (1) follows
similarly to the case of locally Jaffard domains,
using the fact that
if $P$ is an associated prime ideal of $D$,
then $PD[X]_N$ is an associated prime ideal of $D[X]_N$
\cite[Lemma 2.2(2)]{KT22}.
\end{proof}

\section{Specific integral domains related to star-operations arising from Anderson rings}\label{Sec4}

In this section, we investigate some specific integral domains related to star-operations.
To help readers better understand this section,
we have devoted the following subsection to reviewing
some definitions and notation related to star-operations.

\subsection{Preliminaries: definitions and notation}

We begin by recalling some of the most well-known nontrivial star-operations and related concepts.

Let $D$ be an integral domain with quotient field $K$.
Let ${\bf F}(D)$ be the set of nonzero fractional ideals of $D$.
For $I \in {\bf F}(D)$,
we define a fractional ideal $I^{-1}:= \{ x \in K \mid xI \subseteq D \}$.
The mapping on ${\bf F}(D)$ defined by
$I \mapsto I_v := (I^{-1})^{-1}$ is called the {\it $v$-operation} on $D$;
the mapping on ${\bf F}(D)$ defined by
$I \mapsto I_t:=\bigcup J_v$,
where $J$ ranges over the set of all nonzero finitely generated fractional ideals contained in $I$
is called the {\it $t$-operation} on $D$.
An ideal $J$ of $D$ is called a {\it Glaz-Vasconcelos ideal} (for short, {\it GV-ideal}),
and denoted by $J \in {\rm GV}(D)$, if
$J$ is finitely generated and $J^{-1} = D$.
For each $I \in {\bf F}(D)$,
the {$w$-envelope} of $I$ is the set $\{a \in I \otimes K \,|\, aJ \subseteq I \text{ for some }
J \in {\rm GV}(D)\}$, and denoted by $I_w$.
The mapping on ${\bf F}(D)$ defined by
$I \mapsto I_w$ is called the {\it $w$-operation} on $D$.
A nonzero ideal $I$ of $D$ is \textit{divisorial} (or $v$-\textit{ideal})
(respectively, $t$-\textit{ideal}, $w$-\textit{ideal}) if
$I_v=I$ (respectively, $I_t=I$, $I_w=I$).
In general, for $I \in {\bf F}(D)$,
we have the inclusions $I \subseteq I_w\subseteq I_t\subseteq I_v$,
and these inclusions may be strict, as proved in \cite[Proposition 1.2]{Mi03}.
Consequently, every $v$-ideal is a $t$-ideal, and every $t$-ideal is a $w$-ideal.
For $*=t$ or $w$,
a prime ideal is called a {\it prime $*$-ideal} if it is a $*$-ideal, and
a {\it maximal $*$-ideal} is a $*$-ideal that is maximal among the proper $*$-ideals.
Let $*$-${\rm Max}(D)$ be the set of all maximal $*$-ideals of $D$.
Notice that each height-one prime is a prime $t$-ideal,
$w$-$\mathrm{Max}(D)=t$-$\mathrm{Max}(D)$,
and $t$-$\mathrm{Max}(D)\neq\emptyset$ if $D$ is not a field.

Next, we state the definition of $t$-dimension.
Let $P$ be a prime $t$-ideal and consider a strictly descending chain of prime $t$-ideals
$P = P_0 \supsetneq P_1 \supsetneq \cdots \supsetneq P_{n-1} \supsetneq P_n = (0)$.
The supremum of such $n$ is called the {\it $t$-height} of $P$, and denoted by $t$-${\rm ht}(P)$.
The {\it $t$-$($Krull$)$ dimension} of $D$, denoted by $t \text{-dim}(D)$,
is then defined as  $t$-$\dim(D) = \sup\{t\text{-ht}(P) \,|\, P \in t\text{-Max}(D)\}$.
Clearly, the inequality $t$-$\dim(D)\leq\dim(D)$ always holds.
It is worth noting that
$t$-$\dim(D)=1$ if and only if $t$-$\mathrm{Max}(D)=X^1(D)$, where $X^1(D)$ denotes
the set of all height-one prime ideals of $D$.

Recall that an ideal $I$ of $D$ is {\it $t$-invertible} if $(II^{-1})_t = D$.
Using the concept of $t$-invertibility,
we can define the following two classes of integral domains.
\begin{enumerate}
\item[(1)]
$D$ is a {\it Pr\"ufer $v$-multiplication domain} (for short, {\it P$v$MD}) if
every finitely generated ideal is $t$-invertible.
\item[(2)]
$D$ is a {\it Krull domain} if
every nonzero ideal is $t$-invertible.
\end{enumerate}
Now, we provide the definitions of some Krull-like domains.
To facilitate the presentation of the remaining definitions,
we set some standard statements to avoid repetition.
Consider the following conditions:
\begin{itemize}
\item[\rm(i)] $D=\bigcap_{P\in X^1(D)}D_P$.
\item[\rm(ii)] $D_P$ is a valuation domain for each $P\in X^1(D)$.
\item[\rm(iii)] $D_P$ is a Noetherian domain for each $P\in X^1(D)$.
\item[\rm(iv)] Every nonzero element $x$ of $D$ belongs to only finitely many $P\in X^1(D)$.
\end{itemize}

Based on these statements, we define the following:

\begin{enumerate}
\item $D$ is a \textit{weakly Krull domain} if it satisfies \rm(i) and  \rm(iv).
\item $D$ is an \textit{infra-Krull domain} if it satisfies \rm(i), \rm(iii) and  \rm(iv).
\item $D$ is a \textit{generalized Krull domain} (in the sense of Gilmer \cite{gilmer book})
if it satisfies \rm(i), \rm(ii) and  \rm(iv).
\end{enumerate}

Notable examples of P$v$MDs include (generalized) Krull domains and Pr\"ufer domains.
Note that if $D$ is a Krull domain,
then $D$ is both a generalized Krull domain and an infra-Krull domain; and
if $D$ is either a generalized Krull domain or an infra-Krull domain,
then $D$ is a weakly Krull domain.

In \cite{EB02}, El Baghdadi defined an integral domain $D$ to be 
a {\it generalized Krull domain} if
it is a strongly discrete P$v$MD 
({\it i.e.}, $D_M$ is a strongly discrete valuation domain for every maximal $t$-ideal $M$ of $D$)
and every principal ideal has only finitely many minimal prime ideals.
Equivalently, $D$ is a P$v$MD and
for each prime $t$-ideal $P$ of $D$,
there exists a finitely generated ideal $J$ of $D$ such that $P = \sqrt{J_t}$ and $P \neq (P^2)_t$.
For the sake of avoiding confusion, we will use the abbreviation {\it GK-domain}
to refer to generalized Krull domains in the sense of El Baghdadi.
It is important to point out that the definitions of (classical)
generalized Krull domains (in the sense of Gilmer) and GK-domains  do not imply each other.
In fact, it is known that the $t$-dimension of a GK-domain
is not necessarily equal to $1$ as asserted in \cite[page 3740]{EB02},
while the $t$-dimension of a generalized Krull domain is always $1$.
Moreover, any generalized Krull domain that is not a Krull domain cannot be a GK-domain,
because a GK-domain with $t$-dimension one must be a Krull domain \cite[Theorem 3.11]{EB02}.

Following \cite{KP06}, a nonzero ideal $I$ of $D$ is called a 
$t$-\textit{SFT}-\textit{ideal}
if there exist a finitely generated ideal $J \subseteq I$ 
and a positive integer $n$ such that $a^n \in J_v$ for every 
$a \in I_t$. The integral domain $D$ is called a $t$-\textit{SFT}-\textit{ring} 
if every nonzero ideal of $D$ is a $t$-SFT-ideal. 
It is worth noting that Mori domains serve as examples of $t$-SFT rings
(recall that $D$ is a {\it Mori domain} if
it satisfies ascending chain condition on (integral) divisorial ideals).

Lastly, $D$ is a $v$-\textit{domain} if
every nonzero finitely generated ideal $I$ of $D$ is $v$-invertible,
{\it i.e.}, $(II^{-1})_v=D$ for every nonzero finitely generated ideal $I$ of $D$.
Notice that $v$-domains lie between
completely integrally closed domains and integrally closed domains,
and essential domains are $v$-domains.
In addition, $D$ is called a {\it locally $v$-domain}
if $D_P$ is a $v$-domain for any $P \in {\rm Spec}(D)$.
It is clear that locally essential domains are locally $v$-domains.
Notably, while every locally $v$-domain is a $v$-domain,
the converse is not true
(see Proposition 3.1 and the example provided in Section 3 of \cite{FZ11}).
The class of locally $v$-domains coincides with that of super $v$-domains
as introduced in \cite{Z22}
(recall that $D$ is a {\it super $v$-domain} if
$D_S$ is a $v$-domain for any multiplicative subset $S$ of $D$).

\subsection{Main results and applications}

We start this subsection with the following observation on the relationship between the prime $t$-ideals of $D$ and those of $D[X]_A$.

\begin{remark}\label{section begin remark}
{\rm
Let $D$ be an integral domain.
Assume that $\mathfrak{p}$ is a prime $t$-ideal of $D[X]_A$ such that $\mathfrak{p} \cap D\neq(0)$.
For the sake of simplicity, we set $P := \mathfrak{p} \cap D$.
Since $PD[X]_A \subseteq \mathfrak{p}$,
$(PD[X]_A)_t \subseteq \mathfrak{p}_t = \mathfrak{p}$.
From \cite[Proposition 4.2(3)]{BL24}, 
we deduce that $P_tD[X]_A \cap D \subseteq \mathfrak{p} \cap D = P$.
This implies that $P_t \subseteq P$,
so $P$ is a prime $t$-ideal of $D$.
Thus for any prime $t$-ideal $\mathfrak{p}$ of $D[X]_A$ with $\mathfrak{p} \cap D \neq (0)$,
$\mathfrak{p} \cap D$ is a prime $t$-ideal of $D$.
}
\end{remark}

This remark allows us to consider the possibility of comparing the $t$-dimension of $D$ and that of $D[X]_A$.
Hence we first present the following introductory example.

\begin{example}\label{t-dim example 1}
{\rm
Let $D$ be a one-dimensional quasi-local Noetherian domain which is not integrally closed
(e.g. $\mathbb{F}_2[\![Y^2,Y^3]\!]$ or $\mathbb{F}_2 + Y \mathbb{F}_4[\![Y]\!]$,
where $\mathbb{F}_2$ (respectively, $\mathbb{F}_4$) is a finite field with two elements (respectively, four elements)).
It is clear that $t\mhyphen\dim(D) = \dim(D) = 1$.
Now, recall that $t\mhyphen\dim(D[X]_A) \leq \dim(D[X]_A)$, and
$\dim(D[X]) = 1 + \dim(D)$ since $D$ is a Noetherian domain.
This implies that
$t\mhyphen\dim(D[X]_A) \leq \dim(D[X]_A) = \dim(D[X]) = 1 + \dim(D) =2$,
and thus either $t\mhyphen\dim(D[X]_A) = 1$ or $t\mhyphen\dim(D[X]_A) = 2$.
}
\end{example}

Now, to accurately obtain the $t$-dimension of $D[X]_A$ in Example \ref{t-dim example 1},
we investigate some facts about the $t$-dimension of the Anderson ring.

\begin{lemma}\label{t-Dim 0}
Let $D$ be an integral domain.
Then $t\mhyphen\dim(D) \leq t\mhyphen\dim(D[X]_A) \leq
t\mhyphen\dim(D[X])$.
\end{lemma}

\begin{proof}
Suppose that $t\mhyphen\dim(D) = n$.
Then we have the chain $(0) \subsetneq P_1 \subsetneq \cdots \subsetneq P_n$ of
prime $t$-ideals of $D$.
According to \cite[Corollary 4.3]{BL24},
we obtain the chain $(0) \subsetneq P_1D[X]_A \subsetneq \cdots \subsetneq P_nD[X]_A$ of
prime $t$-ideals of $D[X]_A$.
This implies that $t\mhyphen\dim(D[X]_A) \geq n$, 
and this proves the left inequality. 
For the right inequality, consider a chain 
$(0) \subsetneq \mathfrak{p}_1 \subsetneq \cdots \subsetneq \mathfrak{p}_n$ 
of prime $t$-ideals of $D[X]_A$. 
Then it follows from \cite[Lemma 1.15]{GaRo90} that $(0) \subsetneq \mathfrak{p}_1 \cap D[X] \subsetneq \cdots \subsetneq \mathfrak{p}_n \cap D[X]$  is a chain of prime $t$-ideals of $D[X]$.
Hence $t\mhyphen\dim(D[X]_A) \leq t\mhyphen\dim(D[X])$.
\end{proof}

Recall that $D$ is a {\it UMT-domain} if
every upper to zero in $D[X]$ is a maximal $t$-ideal of $D[X]$. 
It is worth noting that P$v$MDs and one $t$-dimensional strong Mori 
domains are subclasses of UMT-domains \cite[Theorems 3.3 and 3.9]{Chang2020}
(recall that $D$ is a {\it strong Mori domain} if
it satisfies ascending chain condition on $w$-ideals).  
Thus by the next result, we obtain the fact that
$t\mhyphen\dim(D[X]_A) = 1$ for the integral domain $D$ in Example \ref{t-dim example 1}
since every one-dimensional Noetherian domain is a UMT-domain.

\begin{proposition}\label{t-Dim 1}
Let $D$ be a UMT-domain that is not a field.
Then $t\mhyphen\dim(D) = t\mhyphen\dim(D[X]) = t\mhyphen\dim(D[X]_A)$.
\end{proposition}

\begin{proof}
By Lemma \ref{t-Dim 0}, we have that $t\mhyphen\dim(D)\leq t\mhyphen\dim(D[X]_A)\leq t\mhyphen\dim(D[X])$.
Moreover, since $D$ is a UMT-domain, $t\mhyphen\dim(D[X])=t\mhyphen\dim(D)$
by \cite[Proposition 2.11(3)]{Wang09},
and thus the proof is complete.
\end{proof}

\begin{remark}
{\rm
Given a positive integer $n$,
it is easy to obtain an integral domain $D$ such that $\dim(D[X]_A)=n$
by means of the equality $\dim(D[X]_A)=\dim(D[X])$ \cite[Proposition 2.4]{BL24}
(e.g. take any Noetherian domain of dimension $n-1$).
For the $t$-analogue of this previous observation,
let $D$ be an $n$-dimensional Pr\"ufer domain.
Since every Pr\"ufer domain is a UMT-domain,
we obtain that $t\mhyphen\dim(D[X]_A) = t\mhyphen\dim(D)$ by Proposition \ref{t-Dim 1},
and hence $t\mhyphen\dim(D[X]_A) = n$ because in Pr\"ufer domains the $t$-operation is trivial.
Thus we can construct Anderson rings of any given $t$-dimension.
}
\end{remark}

By Lemma \ref{t-Dim 0} and Proposition \ref{t-Dim 1}, we have

\begin{corollary}\label{t-Dim}
Let $D$ be an integral domain that is not a field.
If $t$-$\dim(D[X]_A) = 1$, then $t$-$\dim(D) = 1$.
Moreover, the converse holds when $D$ is a UMT-domain.
\end{corollary}

Following \cite{Sei54}, an integral domain $D$ is said to be 
{\it an F-ring} if $\dim(D)=1$ and $\dim(D[X])=3$. 
For instance,  the cited example in Remark \ref{RemJaffard}
and $D=k + Y L[\![Y]\!]$, where $k$ is a finite field and 
$L$ is a transcendental extension of $k$, are examples of F-rings. 
Also, we note that a one-dimensional domain $D$ is an F-ring 
if and only if $D$ is not a UMT domain (cf. \cite[Corollary 3.6]{HZ89} 
and \cite[Theorem 8]{Sei54}).

The next result shows that the converse of the first argument in Corollary \ref{t-Dim} does not hold in general.

\begin{proposition}\label{F-ring}
Let $D$ be an integral domain.
If $D$ is an F-ring, then $t\mhyphen\dim(D[X]_A) = 2$.
\end{proposition}

\begin{proof}
Suppose that $D$ is an F-ring,
{\it i.e.},
$D$ is a one-dimensional integral domain with $\dim(D[X]) = 3$.
Then it is clear that $t\mhyphen\dim(D) = 1$. As mentioned above, $D$ is not a UMT-domain, and then there exists an upper to zero $\mathfrak{p}$ in $D[X]$ that is not a maximal $t$-ideal. Hence there is a maximal $t$-ideal of $D[X]$ that contains properly $\mathfrak{p}$, say $\mathfrak{m}$.
Then $\mathfrak{m} = MD[X]_A$ for some $M \in t\mhyphen{\rm Max}(D)$
\cite[Theorem 4.5]{BL24}.
Also, since $\mathfrak{p}D[X]_A$ has height-one,
it is a prime $t$-ideal of $D[X]_A$.
Hence we obtain the chain $\mathfrak{p}D[X]_A \subsetneq MD[X]_A$
of prime $t$-ideals.
This implies that $t\mhyphen\dim(D[X]_A) \neq 1$,
and hence $t\mhyphen\dim(D[X]) \neq 1$.
Note that $t\mhyphen\dim(D)\leq t\mhyphen\dim(D[X])\leq 2t\mhyphen\dim(D)$ \cite[Theorem 8.7.26]{WK16},
so $t\mhyphen\dim(D[X]) = 2$.
Thus $2 \leq t\mhyphen\dim(D[X]_A) \leq t\mhyphen\dim(D[X]) = 2$,
consequently, $t\mhyphen\dim(D[X]_A) = 2$.
\end{proof}

In what follows, we determine the possible values of $t\mhyphen\dim(D[X]_A)$ in terms of $t\mhyphen\dim(D)$  when $D$ is an integrally closed domain.

\begin{proposition}\label{t-Dim 2}
Let $D$ be an integrally closed domain that is not a field.
Then $t\mhyphen\dim(D[X]_A) - t\mhyphen\dim(D) \leq 1$.
\end{proposition}

\begin{proof}
By Lemma \ref{t-Dim 0}, we have $t\mhyphen\dim(D)\leq t\mhyphen\dim(D[X]_A)\leq t\mhyphen\dim(D[X])$.
Since $D$ is integrally closed, 
it follows from \cite[Proposition 2.11]{Wang2001} that $t\mhyphen\dim(D[X]) \leq 1 + t\mhyphen\dim(D)$.
Thus $t\mhyphen\dim(D)\leq t\mhyphen\dim(D[X]_A)\leq 1+ t\mhyphen\dim(D)$, and this completes the proof.
\end{proof}

Unfortunately, in general, we are unable to determine the value of $t\mhyphen\dim(D[X]_A)$.
This leads to the following open question.

\begin{question}
{\rm
Let $D$ be an integral domain that is not a field.
Does there exist a positive integer $n$ such that $t\mhyphen\dim(D[X]_A) - t\mhyphen\dim(D) \leq n$?
}
\end{question}

Note that the Krull domains satisfy the conditions (1), (2) and (3) in Remark \ref{trivial results},
so we obtain that $D$ is a Krull domain if and only if $D[X]_A$ is a Krull domain.
From now on, we investigate conditions on $D$ under which
$D[X]_A$ becomes some Krull-like domains.
By Corollary \ref{t-Dim},
we obtain results about Krull-like domains.
Before we examine, we need to know a fact about $t$-SFT-rings.

\begin{proposition}\label{t-SFT-ring}
Let $D$ be an integral domain.
If $D[X]_A$ is a $t$-SFT-ring, then so is $D$.
In addition, the converse holds when $D$ is integrally closed.
\end{proposition}

\begin{proof}
Assume that $D[X]_A$ is a $t$-SFT-ring and let $P$ be a prime $t$-ideal of $D$.
Then $PD[X]_A$ is a prime $t$-ideal of $D[X]_A$,
and so there exists a finitely generated ideal $J \subseteq PD[X]_A$
and a positive integer $k$ such that $f^k \in J_v$ for all $f \in (PD[X]_A)_t$.
We may assume that
there exist $f_1,\dots,f_n \in PD[X]$ such that
$J = f_1D[X]_A + \cdots + f_n D[X]_A$.
Hence we obtain
\begin{eqnarray*}
J &=& f_1D[X]_A + \cdots + f_n D[X]_A\\
&\subseteq& c(f_1)D[X]_A + \cdots + c(f_n)D[X]_A \\
&\subseteq& PD[X]_A.
\end{eqnarray*}
Now, let $I = c(f_1) + \cdots + c(f_n)$.
As $J \subseteq ID[X]_A \subseteq PD[X]_A$,
$f^k \in (ID[X]_A)_v = I_vD[X]_A$ for any $f \in PD[X]_A$.
Hence for any $a \in P$,
$a^k \in I_vD[X]_A$.
This implies that $a^k \in I_v$ for any $a \in P$.
Thus $P$ is a $t$-SFT ideal,
and consequently, $D$ is a $t$-SFT-ring. 
When $D$ is assumed to be integrally closed, the converse follows directly from \cite[Theorem 2.13 and Proposition 2.3]{KP06}.
\end{proof}

Now, we state some well-known facts about Krull-like domains.
\begin{itemize}
\item Weakly Krull domains are exactly a one $t$-dimensional domains 
with finite $t$-character \cite[Lemma 2.1(1)]{AMZ92}.
\item Infra-Krull domains are exactly
a one $t$-dimensional strong Mori domain \cite[Page 199]{Kim 2011}.
\item Generalized Krull domains are precisely
a one $t$-dimensional P$v$MD with finite $t$-character \cite[Theorems 3 and 5]{GT18}.
\item GK-domains are exactly a P$v$MD that is also a 
$t$-SFT-ring \cite[Theorem 2.5]{KP06}.
\end{itemize}
By combining Corollary \ref{t-Dim} and Proposition \ref{t-SFT-ring} with the above facts, we obtain

\begin{corollary}\label{Krull-like domains}
Let $D$ be an integral domain.
Then the following assertions hold.
\begin{enumerate}
\item[(1)]
If $D[X]_A$ is a weakly Krull domain,
then so is $D$.
Moreover, the converse holds when $D$ is a UMT-domain.
\item[(2)]
$D[X]_A$ is an infra-Krull domain if and only if
$D$ is an infra-Krull domain.
\item[(3)]
$D[X]_A$ is a generalized Krull domain if and only if
$D$ is a generalized Krull domain.
\item[(4)]
$D[X]_A$ is a GK-domain if and only if
$D$ is a GK-domain.
\end{enumerate}
\end{corollary}

\begin{proof}
(1) This result follows directly from Corollary \ref{t-Dim} and
\cite[Proposition 4.9]{BL24}.

(2) The direct implication follows directly from Corollary \ref{t-Dim} and
\cite[Corollary 4.12]{BL24}.
The converse follows directly from \cite[Theorem 4.3]{PPH05} and \cite[Lemma 2]{MZ92}.

(3) The direct implication follows directly from Corollary \ref{t-Dim} and
\cite[Proposition 4.9]{BL24} with the fact that
$D$ is a P$v$MD if and only if $D[X]_A$ is a P$v$MD (since
the P$v$MDs satisfy the conditions (1), (2) and (3) in Remark \ref{trivial results}).
For the converse, suppose that $D$ is a generalized Krull domain,
{\it i.e.}, $D$ is a one $t$-dimensional P$v$MD with finite $t$-character.
Then $D[X]_A$ is a P$v$MD with finite $t$-character \cite[Proposition 4.9]{BL24}.
Now, since every P$v$MD is a UMT-domain,
$D[X]_A$ is a one $t$-dimensional domain
by Corollary \ref{t-Dim}.
Thus $D[X]_A$ is a generalized Krull domain.

(4) This result follows directly from Proposition \ref{t-SFT-ring} with
the fact that $D$ is a P$v$MD if and only if $D[X]_A$ is a P$v$MD.
\end{proof}

Notice that the converse of (1) in Corollary \ref{Krull-like domains} does not hold in general.
Indeed, if $D$ is a quasi-local one-dimensional domain, then $D$ is a weakly Krull domain,
but $D[X]_A$ is not since $t\mhyphen\dim(D[X]_A) = 2$ by Proposition \ref{F-ring}.

At the end of this section,
we investigate the transfer of the (locally) $v$-domain property between $D$ and $D[X]_A$.
Specifically, we give a full characterization of Anderson rings that are locally $v$-domains.

\begin{proposition}
Let $D$ be an integral domain.
If $D[X]_A$ is a $v$-domain, then so is $D$. 
\end{proposition}

\begin{proof}
Assume that $D[X]_A$ is a $v$-domain and
let $I$ be a nonzero finitely generated ideal of $D$.
Then $ID[X]_A$ is a nonzero finitely generated ideal of $D[X]_A$.
Hence $((II^{-1})D[X]_A)_v=D[X]_A$.
This implies that $(II^{-1})_vD[X]_A=D[X]_A$ \cite[Proposition 4.2(1)]{BL24}.
Consequently, $(II^{-1})_v=D$,
and thus $D$ is a $v$-domain.
\end{proof}

\begin{lemma}\label{LemLocV}
Let $D$ be an integral domain and let $S$ be a multiplicatives subset of $D$.
If $D$ is a locally $v$-domain,
then so is $D_S$.
\end{lemma}

\begin{proof}
This follows from the fact that $(D_S)_{PD_S}=D_P$ for any prime ideal $P$ of $D$.
\end{proof}

\begin{theorem}\label{Locally_v_dom}
Let $D$ be an integral domain.
Then $D[X]_A$ is a locally $v$-domain
 if and only if
$D$ is a locally $v$-domain.
\end{theorem}

\begin{proof}
Assume that $D[X]_A$ is a locally  $v$-domain. 
By \cite[Proposition 3.4]{FZ11},  it suffices to prove that $D_P$ 
is a $v$-domain for every associated prime ideal $P$ of $D$. 
Let $P$ be an associated prime ideal of $D$. According to 
Lemma \ref{PrAssPri}, $PD[X]_A$ is an associated prime 
ideal of $D[X]_A$, so $(D[X]_A)_{PD[X]_A}=D_P[X]_{N_P}$ is 
a $v$-domain. Since $P$ is an associated prime ideal of $D$, 
the maximal ideal of $D_P$ is an associated prime ideal 
by \cite[Lemma 1]{BH74}, and hence it is a $t$-ideal 
({\it{i.e.}}, $D_P$ is a $t$-local domain). 
Hence by \cite[Corollary 4]{Z22}, it follows that $D_P$ 
is a $v$-domain. Therefore $D$ is a locally $v$-domain. 
Conversely, if $D$ is a locally $v$-domain, then 
by \cite[Corollary 5]{Z22}, $D[X]$ is also a locally $v$-domain,
 and thus $D[X]_A$ inherits this property by Lemma \ref{LemLocV}.
\end{proof}

\end{document}